\documentclass[11pt]{amsart}
\usepackage[english]{babel}
\usepackage[iso-8859-7]{inputenc}
\usepackage{latexsym}
\usepackage{amsmath, amsthm, amssymb}
\usepackage{amssymb}
\usepackage{amsmath}
\usepackage{amsfonts}
\usepackage{amsthm}
\usepackage[all]{xy}
\usepackage{indentfirst}
\input{epsf.tex}

\newtheorem{theorem}{Theorem}[section]
\newtheorem{lemma}[theorem]{Lemma}
\newtheorem{proposition}[theorem]{Proposition}

\theoremstyle{definition}
\newtheorem{definition}[theorem]{Definition}

\theoremstyle{remark}

\numberwithin{equation}{section}

\newcommand{\cK}{\mathcal{K}}

\newcommand{\la}{\lambda}

\newcommand{\asc}{{\rm asc}\,}
\newcommand{\codim}{{\rm codim}\,}
\def\e{\varepsilon}
\def\si{\sigma}

\newcommand{\NN}{\mathbb{N}}
\newcommand{\CC}{\mathbb{C}}

\def\toind#1{\smash{\mathop{\longrightarrow}\limits^{#1}}}

\begin{document}

\title{Invariant half-spaces for rank-one perturbations}

\author{V. M\"uller}

\address{Institute of Mathematics,
Czech Academy of Sciences,
ul. \v Zitna 25, Prague,
 Czech Republic}
\email{muller@math.cas.cz}

\date{}

\thanks{The research has been supported by grant No. 20-31529X of GA CR and RVO:67985840}

\begin{abstract}
If $T$ is a bounded linear operator acting on an infinite-dimensional Banach space and $\e>0$, then there exists and operator $F$ of rank at most one with $\|F\|<\e$ such that $T-F$ has an invariant subspace of infinite dimension and codimension. This improves results of Tcaciuc and other authors.
\end{abstract}
\maketitle

\section{Introduction}
The invariant subspace problem is the most important problem in operator theory. It is the question whether each bounded linear operator on a complex Banach space has a nontrivial closed invariant subspace. The problem is still open for operators on Hilbert spaces, or more generally, on reflexive Banach spaces. In the class of non-reflexive Banach spaces negative examples were given by Enflo \cite{E} and Read \cite{R}.

It is easy to see that each operator on a non-separable Banach space has a nontrivial invariant subspace. Similarly, all operators on a finite-dimensional Banach space of dimension at least two have eigenvalues, and so nontrivial invariant subspaces. So the question makes sense only in separable infinite-dimensional Banach spaces.

Inspired by the invariant subspace problem, the following question was studied intensely: given a Banach space operator $T$, does there exists a "small" perturbation $F$ such that $T-F$ has an invariant subspace?

It is easy to see that for each bounded linear operator $T$ on a Banach space $X$ there exists a rank-one operator $F$ such that $T-F$ has a one-dimensional invariant subspace. Indeed, take any non-zero vector $x\in X$ and a rank-one operator $F$ on $X$ such that $Fx=Tx$. Then $(T-F)x=0$ and so $T-F$ has the one-dimensional invariant subspace generated by $x$.

So the proper question is: does every operator $T$ have a "small" perturbation $F$ such that $T-F$ has an invariant subspace of infinite dimension and codimension?

For short, closed subspaces of infinite dimension and codimension are called half-spaces.

The first result in this direction was proved by Brown and Pearcy \cite{BP}:

\begin{theorem}
Let $T$ be an operator on a separable infinite-dimensional Hilbert space $H$ and let $\e>0$. Then there exists a compact operator $K$ on $H$ such that $\|K\|<\e$ and $T-K$ has an invariant half-space.
\end{theorem}

The question has been then studied by a number of authors, see e.g. \cite{APTT}, \cite{MPR}, \cite{PT}, \cite{T}.
The research culminated by \cite{T}, where the following results were proved.

\begin{theorem}
Let $T$ be an operator on an infinite-dimensional Banach space $X$. Then:
\begin{itemize}
\item[(i)] {\rm (\cite{T}, Theorem 1.1)} there exists an operator $F$ of rank at most one such that $T-F$ has an invariant half-space;
\item[(ii)] {\rm(\cite{T}, Theorem 4.3)} if $\e>0$, then there exists an operator $F$ of finite rank with $\|F\|<\e$ such that $T-F$ has an invariant half-space.
\end{itemize}
\end{theorem}

In this note we complete and unify the above results and show that for any operator $T$ acting on an infinite-dimensional  Banach space $X$ and $\e>0$ there exists an operator $F$ of rank at most one with $\|F\|<\e$ such that $T-F$ has an invariant half-space. The proof uses modified techniques from \cite{PT} and \cite{T}.

\section{Preliminaries}
For a (complex) Banach space $X$, we denote by $X^*$ its dual. 

If $M\subset X$ is a subset, then the annihilator $M^\perp$ is defined by
$$
M^\perp=\{x^*\in X^*: \langle m,x^*\rangle=0\hbox{ for all }m\in M\}.
$$
Clearly $M^\perp$ is a $w^*$-closed subspace of $X^*$.

Similarly, for a subset $M'\subset X^*$ define the preannihilator ${}^\perp M'$ by
$$
{}^\perp M'=\{x\in X: \langle x,m^*\rangle=0 \hbox{ for all } m^*\in M'\}.
$$
Clearly ${}^\perp M'$ is a weakly closed (and so closed) subspace of $X$.

A sequence $(x_n)_{n=1}^\infty$ in $X$ is called basic if any vector $x\in\bigvee_{n=1}^\infty x_n$ can be written uniquely as $x=\sum_{n=1}^\infty \alpha_nx_n$ for some complex coefficients $\alpha_n$. Then there exist functionals $x_n^*\in X^*\quad(n\in\NN)$ such that $\langle x_n,x^*_j\rangle=\delta_{n,j}$ (the Kronecker symbol) for all $n,j\in\NN$ and
$\sup\{\|x^*_j\|:j\in\NN\}<\infty$.

Recall that any infinite-dimensional Banach space contains a basic sequence. If $(x_n)_{n=1}^\infty$ is a basic sequence in a Banach space $X$ and $A\subset\NN$ an infinite subset such that $\NN\setminus A$ is infinite then it is easy to see that $\bigvee_{n\in A}x_n$ is a subspace of infinite dimension and codimension. In particular, in any infinite-dimensional Banach space there is a half-space.

The basic result about the existence of basic sequences is the following
 criterion, see \cite{KP} or \cite{AK}, Theorem 1.5.6. Recall that a sequence of vectors in a Banach space is called semi-normalized if it is bounded and bounded away from zero.

\begin{theorem}\label{crit1}{\rm (Kadets-Pe\l czy\'nski)} 
Let $(x_n)_{n=1}^\infty$ be a semi-normalized sequence in a Banach space $X$. Then the following conditions are equivalent:
\begin{itemize}
\item[(i)] $(x_n)_{n=1}^\infty$ fails to contain a basic subsequence;
\item[(ii)] the weak closure $\{x_n:n\in\NN\}^{-w}$ is weakly compact and fails to contain $0$.
\end{itemize}
\end{theorem}

A dual version of this criterion can be found in \cite{GM} or \cite{JR}, Theorem III.1 and Remark III.1.

\begin{theorem}\label{crit2}
If $(x^*_n)_{n=1}^\infty$ is a semi-normalized sequence in a dual Banach space $X^*$ and $0$ is a weak${}^*$-cluster point of
$\{x^*_n\}_{n=1}^\infty$ then there exist a basic subsequence $\{y^*_k\}_{k=1}^\infty$ of $(x^*_n)_{n=1}^\infty$ and a bounded sequence $(y_k)_{k=1}^\infty$ in $X$ such that $\langle y_k,y^*_j\rangle=\delta_{k,j}$ for all $k,n\in\NN$.
\end{theorem}

Denote by $B(X)$ the algebra of all (bounded linear) operators on a Banach space $X$.

Let $T\in B(X)$. Denote by $N(T)$ the kernel of $T$, $N(T)=\{x\in X: Tx=0\}$, and by $R(T)=TX$ the range of $T$.
Write $R^\infty(T)=\bigcap_{k=1}^\infty R(T^k)$. Clearly $R^\infty(T)$ is a (not necessarily closed) linear manifold.

Clearly $N(T)\subset N(T^2)\subset N(T^3)\subset\cdots$. Denote by $\asc T$ the ascent of $T$, $\asc T=\min\{k: N(T^{k+1})=N(T^k)\}$ (if no such $k$ exists then we set $\asc T=\infty$). It is easy to see that if $\asc T=k<\infty$ then $N(T^j)=N(T^k)$ for all $j\ge k$.

Denote by $\cK(X)$ the closed two-sided ideal of all compact operators on $X$. For $T\in B(X)$ let $\|T\|_e$ be the essential norm of $T$, $\|T\|_e=\inf\{\|T+K\|:K\in\cK(X)\}$.
Let $\si_e(T)$ be the essential spectrum of an operator $T\in B(X)$, $\si_e(T)=\{\la\in\CC: T-\la\hbox{ is not Fredholm}\}$. It is well known that $\si_e(T)$ is the spectrum of the class $T+\cK(X)$ in the Calkin algebra $B(X)/\cK(X)$.

\section{Main result}
For short, we use the following definition.

\begin{definition}
Let $T$ be an operator acting on an infinite-dimensional separable Banach space $X$. We say that $T$ has property (P) if for every $\e>0$ there exists an operator $F\in B(X)$ of rank at most one such that $\|F\|<\e$ and $T-F$ has an invariant half-space.
\end{definition}

\begin{proposition}\label{P1}
Let $X$ be a separable infinite-dimensional Banach space, let $T\in B(X)$, $0\in\si_e(T)\cap\partial\si(T)$ and $\asc T<\infty$. Then $T$ has property (P).
\end{proposition}

\begin{proof}
If $\dim N(T)=\infty$ then any half-subspace of $N(T)$ is invariant for $T$, and so $T$ has property (P).

So we may assume that $\dim N(T)<\infty$. Let $\e>0$. Let $k=\asc T<\infty$ and $E=N(T^k)$. Then $E$ is a finite-dimensional subspace of $X$, $\dim E\le k\cdot\dim N(T)$. Let $M\subset X$ be a complement of $E$, $X=E\oplus M$. Let $P_M$ be the projection onto $M$ along $E$. Then $P_E:=I-P_M$ is the projection onto $E$ along $M$.

Find a sequence $(\la_n)\subset \CC\setminus\si(T)$ such that $\la_n\to 0$. Since $0\in\si_e(T)$, we have $\lim_{n\to\infty}\|(T-\la_n)^{-1}\|_e=\infty$. Consequently,
$$
\lim_{n\to\infty}\|P_M(T-\la_n)^{-1}\|=\infty.
$$
By the Banach-Steinhaus uniform boundedness theorem, there exists a vector $u\in X$, $\|u\|=1$ such that
$$
\sup\bigl\{\|P_M(T-\la_n)^{-1}u\|:n\in\NN\bigr\}=\infty.
$$
Without loss of generality we may assume that $\lim_{n\to\infty}\bigl\|P_M(T-\la_n)^{-1}u\bigr\|=\infty$.
For $n\in\NN$ set 
$$
x_n=\frac{P_M(T-\la_n)^{-1}u}{\|P_M(T-\la_n)^{-1}u\|}.
$$
We have $\|x_n\|=1$ and 
\begin{align}\label{L1}
Tx_n&=&
\la_n x_n+(T-\la_n)x_n=
\la_nx_n+\frac{(T-\la_n)(I-P_E)(T-\la_n)^{-1}u}{\|P_M(T-\la_n)^{-1}u\|}\cr
&=&
\la_nx_n+\frac{u}{\|P_M(T-\la_n)^{-1}u\|}-\frac{(T-\la_n)P_E(T-\la_n)^{-1}u}{\|P_M(T-\la_n)^{-1}u\|},
\end{align}
where $\la_nx_n+\frac{u}{\|P_M(T-\la_n)^{-1}u\|}\to 0$ as $n\to\infty$ and
$$
\frac{(T-\la_n)P_E(T-\la_n)^{-1}u}{\|P_M(T-\la_n)^{-1}u\|}\in E
$$
since $TE\subset E$.

We show that the sequence $(x_n)_{n=1}^\infty$ has a basic subsequence. Suppose the contrary.
By Theorem \ref{crit1}, $\{x_n:n\in\NN\}^{-w}$ is weakly compact and does not contain $0$. By the Eberlein-Smulian theorem, there exists a weakly convergent subsequence $(x_k)$ of $(x_n)$, $x_k\toind{w} x$ and $x\ne 0$.
Then
$Tx_k\toind{w}Tx$ and $Tx\in E=N(T^k)$ 
by (\ref{L1}). So $x\in N(T^{k+1})=N(T^k)=E$. By definition, $x_k\in M$ for all $k$, and so $x\in M$. Hence $x=0$, a contradiction.

So the set $\{x_n: n\in\NN\}$ contains a basic sequence. By passing to a subsesquence if necessary we may assume that $(x_n)_{n=1}^\infty$ is a basic sequence in $M$. Let $(x^*_n)_{n=1}^\infty\subset M^*$ be the corresponding biorthogonal sequence,
$\langle x_n,x^*_j\rangle=\delta_{n,j}$ for all $n,j\in\NN$.

Set $y_n^*=x^*_nP_M\in X^*\quad(n\in\NN)$. Then $y^*_n\in E^\perp$, $\langle x_n,y^*_j\rangle=\delta_{n,j}$ for all $n,j\in\NN$ and $c:=\sup\{\|y^*_j\|:j\in\NN\}<\infty$.
Without loss of generality we may assume that $\sum_{n=1}^\infty \|P_M(T-\la_{2n})^{-1}u\|^{-1}<\e/c$.

Set $F= u\otimes \Bigl(\sum_{n=1}^\infty \frac{y_{2n}^*}{\|P_M(T-\la_{2n})u\|}\Bigr)$. Then $F$ is an operator of rank one  and 
$$
\|F\|=\|u\|\cdot\Bigl\|\sum_{n=1}^\infty \frac{y_{2n}^*}{\|P_M(T-\la_{2n})u\|}\Bigr\|\le
\sum_{n=1}^\infty \frac{c}{\|P_M(T-\la_{2n})^{-1}u\|}<\e.
$$
Let $L=\Bigl(\bigvee_{n=1}^\infty x_{2n}\Bigr)\vee E$. 
Clearly $\dim L\ge \dim \bigvee_{n=1}^\infty x_{2n}=\infty$. Furthermore, $\codim \bigvee_{n=1}^\infty x_{2n}=\infty$ and since $\dim E<\infty$, we have $\codim L=\infty$.
Hence $L$ is a half-space and $(T-F)E=TE\subset E\subset L$.
Furthermore, for $n\in\NN$ we have
$$
(T-F)x_{2n}=
\la_{2n}x_{2n}+\frac{u}{\|P_M(T-\la_{2n})^{-1}u\|}-\frac{(T-\la_{2n})P_E(T-\la_{2n})^{-1}u}{\|P_M(T-\la_{2n})^{-1}u\|}
$$
$$
-\frac{u}{\|P_M(T-\la_{2n})^{-1}u\|}
$$
$$
=
\la_{2n}x_{2n}-\frac{(T-\la_{2n})P_E(T-\la_{2n})^{-1}u}{\|P_M(T-\la_{2n})^{-1}x_{2n}\|}\in L.
$$

\end{proof}

The dual result is also true.

\begin{proposition}\label{P2}
Let $X$ be a separable infinite-dimensional Banach space, let $T\in B(X)$, $0\in\si_e(T)\cap\partial\si(T)$ and $\asc T^*<\infty$. Then $T$ has property (P).
\end{proposition}

\begin{proof}
If $\dim N(T^*)=\infty$ then $\codim \overline{R(T)}=\infty$. If $\dim R(T)=\infty$ then $\overline{R(T)}$ is a half-space invariant for $T$. If $\dim R(T)<\infty$ then $\dim N(T)=\infty$ and any half-subspace of $N(T)$ is invariant for $T$. So $T$ has property (P).

So we may assume that $\dim N(T^*)<\infty$. Let $\e>0$. Let $k=\asc T^*<\infty$. Let $E'=N(T^{*k})$. Then $E'$ is a finite-dimensional subspace of $X^*$.

We have $\overline{R(T^k)}={}^\perp E'$. So $\codim \overline{R(T^{k})}<\infty$. Let $E\subset X$ be a complement of $\overline{R(T^k)}$, $X=E\oplus \overline{R(T^k)}$. Let $M'=E^\perp$. Then $M'$ is a $w^*$-closed subspace of $X^*$ and $X^*=M'\oplus E'$. Let $P_{M'}$ be the projection onto $M'$ along $E'$. Then $P_{E'}:=I-P_{M'}$ is the projection onto $E'$ along $M'$.

Find a sequence $(\la_n)\subset \CC\setminus\si(T)=\CC\setminus\si(T^*)$ such that $\la_n\to 0$. Since $0\in\si_e(T)=\si_e(T^*)$, we have $\lim_{n\to\infty}\|(T^*-\la_n)^{-1}\|_e=\infty$. Consequently,
$$
\lim_{n\to\infty}\|P_{M'}(T^*-\la_n)^{-1}\|=\infty.
$$
By the Banach-Steinhaus uniform boundedness theorem, there exists a vector $u^*\in X^*$, $\|u^*\|=1$ such that
$$
\sup\bigl\{\|P_{M'}(T^*-\la_n)^{-1}u^*\|: n\in\NN\bigr\}=\infty.
$$
Without loss of generality we may assume that $\lim_{n\to\infty}\bigl\|P_{M'}(T^*-\la_n)^{-1}u^*\bigr\|=\infty$.
For $n\in\NN$ set 
$$
y^*_n=\frac{P_{M'}(T^*-\la_n)^{-1}u^*}{\|P_{M'}(T^*-\la_n)^{-1}u^*\|}.
$$
We have $\|y^*_n\|=1$ and 
$$
T^*y^*_n=
\la_n y^*_n+(T^*-\la_n)y^*_n=
\la_ny^*_n+\frac{(T^*-\la_n)(I-P_{E'})(T^*-\la_n)^{-1}u^*}{\|P_{M'}(T^*-\la_n)^{-1}u^*\|}
$$
$$
=
\la_ny^*_n+\frac{u^*}{\|P_{M'}(T^*-\la_n)^{-1}u^*\|}-\frac{(T^*-\la_n)P_{E'}(T^*-\la_n)^{-1}u^*}{\|P_{M'}(T^*-\la_n)^{-1}u^*\|},
$$
where $\la_ny^*_n+\frac{u^*}{\|P_{M'}(T^*-\la_n)^{-1}u^*\|}\to 0$ as $n\to\infty$ and
$$
\frac{(T^*-\la_n)P_{E'}(T^*-\la_n)^{-1}u^*}{\|P_{M'}(T^*-\la_n)^{-1}u^*\|}\in E'
$$
since $T^*E'\subset E'$.

Since $X$ is a separable Banach space, the closed unit ball in $X^*$ with the $w^*$-topology is metrizable and compact. So $(y^*_n)$ has a $w^*$-convergent subsequence. Without loss of generality we may assume that $y^*_n\toind{w^*} y^*$. So $T^*y^*_n\toind{w^*} T^*y^*\in E'=N(T^{*k})$. Hence $y^*\in N(T^{*k+1})=N(T^{*k})=E'$. However, clearly $y^*\in M'$, and so $y^*=0$.
By Theorem \ref{crit2}, $(y^*_n)_{n=1}^\infty$ contains a basic subsequence. Without loss of generality we may assume that $(y^*_n)_{n=1}^\infty$ is basic. Let $(y_n)_{n=1}^\infty\subset X$ be a bounded sequence satisfying $\langle y_n,y^*_j\rangle=\delta_{n,j}\quad(n,j\in\NN)$.

Let $Q:X=E\oplus\overline{R(T^k)}\to\overline{R(T^k)}$ be the canonical projection onto $\overline{R(T^k)}$ along $E$. Let $x_n=Qy_n\quad(n\in\NN)$. Then $x_n\in \overline{R(T^k)}$ for all $n$, $\langle x_n,y^*_j\rangle=\delta_{n,j}\quad(n,j\in\NN)$ and $c:=\sup\{\|x_n\|:n\in\NN\}<\infty$.


By passing to a subsequence if necessary we can assume that
$$
\sum_{n=1}^\infty\frac{1}{\|P_{M'}(T^*-\la_n)^{-1}u^*\|}<\frac{\e}{c}.
$$

Define operator $F= \Bigl(\sum_{n=1}^\infty \frac{x_{2n}}{\|P_{M'}(T^*-\la_{2n})u^*\|}\Bigr) \otimes u^*$. Then $F$ is an operator of rank one and 
$$
\|F\|=\|u^*\|\cdot\Bigl\|\sum_{n=1}^\infty \frac{x_{2n}}{\|P_{M'}(T^*-\la_{2n})u^*\|}\Bigr\|\le
\sum_{n=1}^\infty \frac{c}{\|P_{M'}(T^*-\la_{2n})^{-1}u^*\|}<\e.
$$
Let 
$$
L'=\Bigl(\bigvee_{n=1}^\infty y^*_{2n}\Bigr) \vee E'\subset X^*
$$
and 
$$
L={}^\perp L'={}^\perp\Bigl(\bigvee_{n=1}^\infty y^*_{2n}\Bigr) \cap {}^\perp E'.
$$
Clearly ${}^\perp\bigl(\bigvee_{n=1}^\infty y^*_{2n}\bigr)\supset \{x_{2j+1}:j\in\NN\}$. So $\dim {}^\perp\bigl(\bigvee_{n=1}^\infty y^*_{2n}\bigr)=\infty$. Moreover, $\codim {}^\perp E'<\infty$, and so $\dim L=\infty$.

Similarly, $x_{2j}\notin {}^\perp\bigl(\bigvee_{n=1}^\infty y^*_{2n}\bigr)\supset L$. So $\codim L=\infty$ and $L$ is a half-space.

We show that $(T-F)L\subset L$.
Let $z\in L={}^\perp\bigl(\bigvee_{n=1}^\infty y^*_{2n}\bigr)\cap {}^\perp E'.$ To show that $(T-F)z\in L$ we must show that
$$
\langle (T-F)z,y^*_{2j}\rangle=0
$$
for all $j\in\NN$ and
$$
\langle (T-F)z,y^*\rangle=0
$$
for all $y^*\in E'$.

We have
$$
\langle (T-F)z,y^*_{2n}\rangle=
\langle z,T^*y^*_{2n}\rangle-\langle Fz,y^*_{2n}\rangle
$$
$$
=
\Bigl\langle z,\la_{2n}y^*_{2n}+\frac{u^*}{\|P_{M'}(T^*-\la_{2n})^{-1}u^*\|}-\frac{(T^*-\la_{2n})P_{E'}(T^*-\la_{2n})^{-1}u^*}{\|P_{M'}(T^*-\la_{2n})^{-1}u^*\|}\Bigr\rangle
$$
$$
-\frac{\langle z,u^*\rangle}{\|P_{M'}(T^*-\la_{2n})^{-1}u^*\|}=0.
$$
Furthermore, for $y^*\in E'$ we have 
$$
\langle (T-F)z,y^*\rangle=
\langle z,T^*y^*\rangle-\langle z,u^*\rangle\cdot \Bigl\langle \sum_{n=1}^\infty\frac{x_{2n}}{\|P_{M'}(T^*-\la_{2n})^{-1}u^*\|},y^*\Bigr\rangle=0
$$
since $T^*y^*\in E'$, $z\in{}^\perp E'$ and $x_{2n}\in {}^\perp E'$ for all $n\in\NN$.
Hence $(T-F)L\subset L$, and so $T$ has property (P).
\end{proof}

\begin{lemma}\label{P3}
Let $T\in B(X)$, $\dim N(T)<\infty$. Then $TR^\infty(T)=R^\infty(T)$.
\end{lemma}

\begin{proof}
Clearly $TR^\infty(T)\subset R^\infty(T)$.

Let $x\in R^\infty(T)$. Suppose on the contrary that $x\notin TR^\infty(T)$.

Let $n=\dim N(T)$. Set $k_0=0$. Since $x\in R^\infty(T)\subset R(T)$, there exists $y_0\in X$ such that $Ty_0=x$. Since $x\notin TR^\infty(T)$, there exists $k_1\in\NN$ such that $y_0\notin R(T^{k_1})$.

Since $x\in R^\infty(T)\subset R(T^{k_1+1})$, there exists $y_1\in R(T^{k_1})$ with $Ty_1=x$. Since $x\notin TR^\infty(T)$, there exists $k_2>k_1$ such that $y_1\notin R(T^{k_2})$. 

Inductively we can find vectors $y_1,\dots,y_n,y_{n+1}$ and numbers $k_1<k_2<\dots k_{n+1},k_{n+2}$ such that
$Ty_j=x$ and $y_j\in R(T^{k_j})\setminus R(T^{k_{j+1}})$ for $j=1,\dots, n+1$.

Set $u_j=y_j-y_0\quad(j=1,\dots,n+1)$. Clearly $Tu_j=0$ for all $j=1,\dots,n+1$. We show that the vectors $u_1,\dots,u_{n+1}$ are linearly independent. Suppose that $\sum_{j=1}^{n+1}\alpha_ju_j=0$ for some coefficients $\alpha_j$. 
We have
$$
0=\sum_{j=1}^{n+1}\alpha_jy_j-y_0\sum_{j=1}^{n+1}\alpha_j,
$$
where $\sum_{j=1}^{n+1}\alpha_jy_j\in R(T^{k_1})$ and $y_0\notin R(T^{k_1})$. So $\sum_{j=1}^{n+1}\alpha_j=0$.

Let $j_0$ be the smallest index such that $\alpha_{j_0}\ne 0$. Then
$$
0=\sum_{j=j_0}^{n+1}\alpha_jy_j\in \alpha_{j_0}y_{j_0}+R(T^{k_{j_0+1}}).
$$
Since $y_{j_0}\notin R(T^{k_{j_0+1}})$, we have $\alpha_{j_0}=0$. So $\alpha_j=0$ for all $j$ and the vectors $u_1,\dots,u_{n+1}$ are linearly independent elements in $N(T)$, a contradiction with the assumption that $\dim N(T)=n$.

Hence $x\in TR^\infty(T)$.
\end{proof}

\begin{theorem}
Let $X$ be an infinite-dimensional Banach space, let $T\in B(X)$ and $\e>0$. Then there exists an operator $F\in B(X)$ of rank at most one such that $\|F\|<\e$ and $T-F$ has an invariant half-space.
\end{theorem}

\begin{proof}
Without loss of generality we may assume that $X$ is separable.

Let $\la\in\si_e(T)$ satisfy $|\la|=\max\{|\mu|:\mu\in\si_e(T)\}$. Then there are only countably many elements $\mu\in\si(T)$ satisfying $|\mu|>|\la|$. So there exists a sequence $(\la_n)\subset \CC\setminus\si(T)$ such that $\la_n\to \la$.

Replacing $T$ by $T-\la$ we may assume without loss of generality that $\la=0$.

So we may assume that $0\in\si_e(T)\cap\partial\si(T)$. 

By Proposition \ref{P1}, we may assume that $\asc T=\infty$. Clearly we may assume that $\dim N(T)<\infty$; otherwise any half-subspace of $N(T)$ is invariant for $T$.  We have
$$
N(T)\supset R(T)\cap N(T)\supset R(T^2)\cap N(T)\supset\cdots,
$$
where $R(T^j)\cap N(T)\ne\{0\}$ for all $j\in\NN$. So there exists $j_0\in \NN$ such that 
$$
R(T^j)\cap N(T)= R(T^{j_0})\cap N(T)
$$ for all $j\ge j_0$. Hence $R(T^{j_0})\cap N(T)= R^\infty(T)\cap N(T)$
and there exists a nonzero vector $y_0\in R^\infty(T)\cap N(T)$.

By Lemma \ref{P3}, we have $TR^\infty(T)=R^\infty(T)$. So we can find inductively vectors $y_j\in R^\infty(T)\quad(j\in\NN)$ such that $Ty_j=y_{j-1}\quad(j\ge 1)$ and $Ty_0=0$.

By Proposition \ref{P2}, we can assume that $\asc T^*=\infty$. Similarly we can find vectors $y^*_0,y^*_1,y^*_2,\dots\in X^*$ such that $y^*_0\ne 0$, $T^*y^*_j=y^*_{j-1}\quad(j\in\NN)$ and $T^*y^*_0=0$.

Let $L=\bigvee_{j=0}^\infty y_j$. Clearly $TL\subset L$. Vectors $y_j\quad(j\ge 0)$ are linearly independent. Indeed, suppose that $\sum_{j=0}^\infty\alpha_jy_j=0$ for some finite sum. Suppose that $\alpha_{j_0}\ne 0$ and $\alpha_j=0$ for all $j> j_0$. Then 
$$
0=T^{j_0}\sum_{j=0}^\infty\alpha_jy_j=\alpha_{j_0}y_{0}.
$$
So $\alpha_{j_0}=0$, a contradiction. So $\dim L=\infty$.

For $j,k=0,1,\dots$ we have
$$
\langle y_j,y^*_k\rangle=0
$$
since $y_j\in R^\infty(T)\subset R(T^{k+1})$ and $y^*_k\in N(T^{*k+1})$.
So $L^\perp\supset\bigvee\{y^*_k:k=0,1,\dots\}$ where the vectors $y^*_k$ are linearly independent as above. Hence $\dim L^\perp=\infty$ and so $\codim L=\infty$.

Hence $T$ has an invariant half-space.
\end{proof}

\end{document}